\documentclass[11pt]{amsart}
\usepackage{pifont}
\usepackage{amsthm}
\usepackage{amsfonts}
\usepackage{amssymb}
\usepackage[mathscr]{euscript}
\usepackage[all]{xy}
\usepackage[dvips]{graphics}
\usepackage{amsmath}
\usepackage[dvips,final]{graphicx}
\usepackage[dvips]{geometry}
\usepackage{color}
\usepackage{epsfig}
\usepackage{latexsym}
\usepackage{subcaption}

\graphicspath{{diagrams/}}

\newtheorem{theorem}{Theorem}

\newtheorem{proposition}[theorem]{Proposition}
\newtheorem{corollary}[theorem]{Corollary}
\newtheorem{conjecture}{Conjecture}

\theoremstyle{definition}

\newtheorem{example}{Example}[subsection]

\theoremstyle{remark}

\title{Virtual Parity Alexander Polynomial}
\author{Heather A. Dye}
\author{Aaron Kaestner}

\begin{document}
\begin{abstract} In this paper, we define the parity virtual Alexander polynomial following the work of BDGGHN \cite{boden} and  Kaestner and Kauffman \cite{KaestnerKauffman}. The properties of this invariant are explored and some examples are computed. In particular, the invariant demonstrates that many virtual knots can not be unknotted by crossing change on only odd crossings.
\end{abstract}
\maketitle

\section{Introduction}
The motivation is to create an Alexander-type polynomial that gives lower bounds on the number of virtual and odd crossings in a virtual link. Our method involves constructing a group that respects the Reidemeister and virtual Reidemeister moves and that differentiates between odd and even classical crossings, as well as virtual crossings. We use classical techniques to construct an Alexander type polynomial. See BDGGHN's work \cite{boden} and  Kitano's survey paper on Alexander polynomials\cite{kitano} for reference. See Kaestner and Kauffman's work \cite{KaestnerKauffman} or KNS  \cite{kaestnernelson}  for reference on using parity with biquandle structures.
Differentiating between even and odd crossings results in a polynomial that gives additional information about the non-planarity of virtual knot.

In Section 2, we review virtual knots. Section 3 introduces the virtual parity group and we derive the virtual parity Alexander module in Section 4 and compute some examples. In Section 5, we examine the properties of the polynomial.
% something about the geometric aspects?

\section{Virtual Knots}

\begin{figure}
\begin{subfigure}{0.49\linewidth}
\[ \begin{array}{c} \scalebox{0.5}{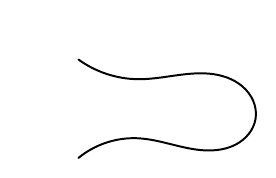} \end{array} \leftrightarrow
\begin{array}{c} \scalebox{0.5}{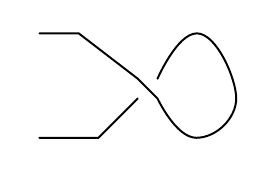} \end{array} \]
\caption{Reidemeister I}
\label{fig:r1move}
\end{subfigure}
\begin{subfigure}{0.49\linewidth}
\[ \begin{array}{c} \scalebox{0.5}{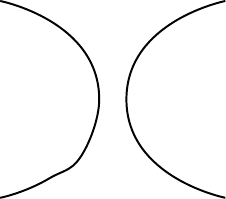} \end{array} \leftrightarrow
\begin{array}{c} \scalebox{0.5}{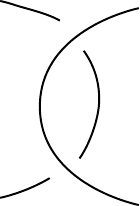} \end{array} \]
\caption{Reidemeister II}
\label{fig:r2move}
\end{subfigure} \\
\begin{subfigure}{0.80\linewidth}
\[ \begin{array}{c} \scalebox{0.5}{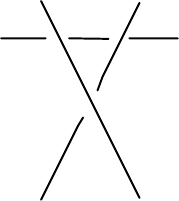} \end{array} \leftrightarrow
\begin{array}{c} \scalebox{0.5}{ 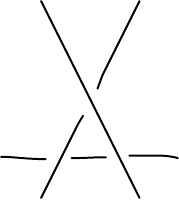} \end{array} \]
\caption{Reidemeister III}
\label{fig:r3move}
\end{subfigure}
\caption{Reidemeister moves}
\label{fig:rmoves}
\end{figure}
A virtual knot diagram is a decorated immersion of $S^1$ into the plane. There are two types of double points: classical crossings (indicated by over/under markings) and virtual crossings (indicated by a circled crossing). Two virtual knot diagrams, $K_1$ and $K_2$, are equivalent if one can be transformed into the other by a sequence of Reidemeister moves and virtual Reidemeister moves \cite{introvkt}.
A virtual knot is an equivalence class of virtual link diagrams determined by the Reidemeister moves and the virtual Reidemeister moves. For convenience, we collectively refer to the Reidemeister and virtual Reidemeister moves as the diagrammatic moves.
\begin{figure}
\begin{subfigure}{0.49\linewidth}
\[ \begin{array}{c} \scalebox{0.5}{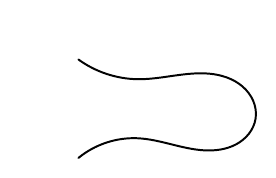} \end{array} \leftrightarrow
\begin{array}{c} \scalebox{0.5}{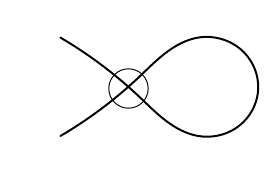} \end{array} \]
\caption{Virtual I}
\label{fig:vr1move}
\end{subfigure}
\begin{subfigure}{0.49\linewidth}
\[ \begin{array}{c} \scalebox{0.5}{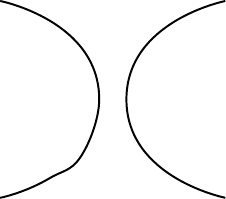} \end{array} \leftrightarrow
\begin{array}{c} \scalebox{0.5}{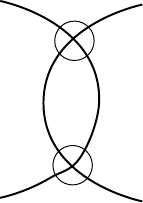} \end{array} \]
\caption{Virtual II}
\label{fig:vr2move}
\end{subfigure} \\
\begin{subfigure}{0.49\linewidth}
\[ \begin{array}{c} \scalebox{0.5}{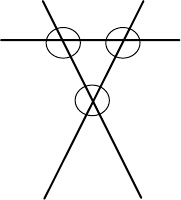} \end{array} \leftrightarrow
\begin{array}{c} \scalebox{0.5}{ 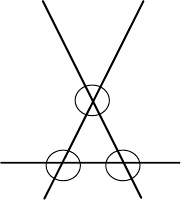} \end{array} \]
\caption{Virtual III}
\label{fig:vr3move}
\end{subfigure}
\begin{subfigure}{0.49\linewidth}
\[ \begin{array}{c} \scalebox{0.5}{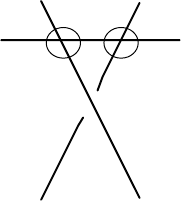} \end{array} \leftrightarrow
\begin{array}{c} \scalebox{0.5}{ 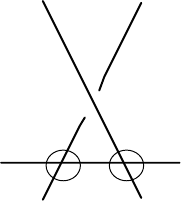} \end{array} \]
\caption{Virtual IV}
\label{fig:vr4move}
\end{subfigure}
\caption{Virtual Reidemeister moves}
\label{fig:vrmoves}
\end{figure}

Equivalently, virtual knots may be defined as a pair $(F \times I, K)$ where $F$ is a compact, oriented surface and $K$ is an embedding of $S^1$ into $F \times I$. Two such pairs, $ (F \times I, K)$ and
$(F' \times I, K')$, are equivalent  if $F \times I$ can be transformed into $F' \times I$ via handle stablization/destablizations
and $K$ can be transformed into $K'$ by isotopy (Reidemeister moves) or Dehn twists of the surface. Equivalence classes of these pairs are in bijective correspondence with virtual knots.  See  the article by Carter, Kamada, and Saito on stable equivalence \cite{carterkamada} or  Kamada and Kamada's article  for a description of abstract link diagrams \cite{aldkamada}.

\subsection{Even and Odd Crossings}
Even and odd crossings arise from the parity of a virtual knot; see Ilyutko, Manturov, and Nikonov \cite{ManturovIlyutkoNikonov}.
To determine if a crossing is either even or odd, we choose a base point in the knot diagram and place a label at each classical crossing. The knot is then traversed and the labels are recorded as encountered.
For example, the knot shown in figure \ref{fig:gausscodeexample} has the
code $abacbc$. This is a simplified version of the Gauss code of the knot given in \cite{introvkt}.

\begin{figure}
\[ \begin{array}{c}\scalebox{0.25} {\fontsize{204pt}{20pt}\selectfont%% Creator: Inkscape inkscape 0.92.4, www.inkscape.org
%% Accompanies image file 'gausscodeexample.pdf' (pdf, eps, ps)
\begingroup%
  \makeatletter%
  \providecommand\color[2][]{%
    \errmessage{(Inkscape) Color is used for the text in Inkscape, but the package 'color.sty' is not loaded}%
    \renewcommand\color[2][]{}%
  }%
  \providecommand\transparent[1]{%
    \errmessage{(Inkscape) Transparency is used (non-zero) for the text in Inkscape, but the package 'transparent.sty' is not loaded}%
    \renewcommand\transparent[1]{}%
  }%
  \providecommand\rotatebox[2]{#2}%
  \newcommand*\fsize{\dimexpr\f@size pt\relax}%
  \newcommand*\lineheight[1]{\fontsize{\fsize}{#1\fsize}\selectfont}%
  \ifx\svgwidth\undefined%
    \setlength{\unitlength}{196.70327777bp}%
    \ifx\svgscale\undefined%
      \relax%
    \else%
      \setlength{\unitlength}{\unitlength * \real{\svgscale}}%
    \fi%
  \else%
    \setlength{\unitlength}{\svgwidth}%
  \fi%
  \global\let\svgwidth\undefined%
  \global\let\svgscale\undefined%
  \makeatother%
  \begin{picture}(1,1.26329479)%
    \lineheight{1}%
    \setlength\tabcolsep{0pt}%
    \put(0,0){\includegraphics[width=\unitlength,page=1]{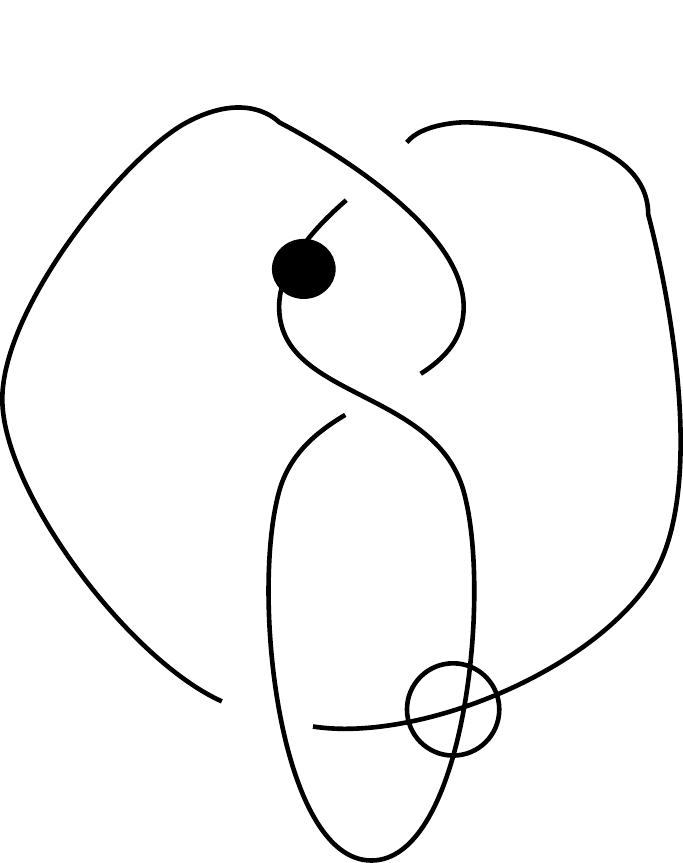}}%
    \put(0.48985955,1.15457984){\color[rgb]{0,0,0}\makebox(0,0)[lt]{\lineheight{0}\smash{\begin{tabular}[t]{l}a\end{tabular}}}}%
    \put(0.74927717,0.63574479){\color[rgb]{0,0,0}\makebox(0,0)[lt]{\lineheight{0}\smash{\begin{tabular}[t]{l}b\end{tabular}}}}%
    \put(0.23553779,0.0141616){\color[rgb]{0,0,0}\makebox(0,0)[lt]{\lineheight{0}\smash{\begin{tabular}[t]{l}c\end{tabular}}}}%
  \end{picture}%
\endgroup%
} \end{array} \]
\caption{Labeled knot diagram}
\label{fig:gausscodeexample}
\end{figure}
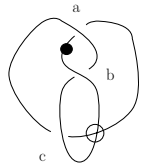

Crossings that are evenly intersticed are \textit{even} crossings. Crossings that are oddly intersticed are \textit{odd} crossings. In the example from figure \ref{fig:gausscodeexample}, the crossings $a$ and $c$ are odd while crossing $b$ is even. In knot diagrams without virtual crossings, every crossing is even.  The parity of a crossing abstractly gives information about the planarity of the diagram.  For more information about even and odd crossings and their interactions with the Reidemeister moves, see Chrisman and Dye \cite{chrismandye}, or Kaestner and Kauffman \cite{KaestnerKauffman}.

\section{The virtual parity group}

The \textit{virtual parity group} of a virtual knot diagram $K$, $PG_K$, is a free group modulo relations determined by the virtual semi-arcs of the diagram and the crossings.  The \textit{virtual semi-arcs} of the diagram are the edges of the diagram which begin and terminate at either a virtual or real crossing. The generators of $PG_K$ are the labels assigned to the semi-arcs and the generators $s,q,$ and $\theta$.
The crossings determine the group relations and the group is invariant under the Reidemeister and virtual Reidemeister moves.
For information about related group structures on virtual knot groups, see BGHNW \cite{bodengroup} and Silver and Williams \cite{silverwilliamsgroup}.

 For a virtual knot diagram $K$ with $n$ total crossings, the  virtual parity group has the form:
\begin{equation*}
PG_K = \lbrace  a_1, a_2,  \ldots a_{2n}, s,q, \theta \vert r_{11}, r_{12}, r_{21}, r_{22}, \ldots
r_{n1},r_{n2}, [s,q], [s, \theta], [q, \theta] \rbrace.
\end{equation*}
The  terms $r_{i1}$ and $r_{i2}$ are obtained from the $i$th crossing.  The commutator terms  $ [a,b]$($ = aba^{-1}b^{-1}$) insure invariance under the diagrammatic moves.
Classical even, classical odd, and virtual crossings determine different relations.

\begin{figure}
\begin{subfigure}{0.32\linewidth}
\[ \begin{array}{c}\scalebox{0.5} {\fontsize{204pt}{20pt}\selectfont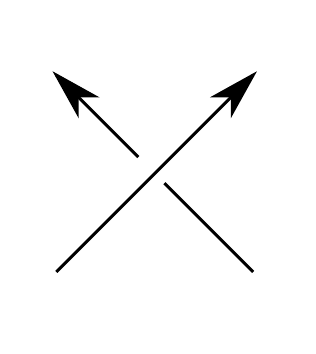} \end{array} \]
\caption{Positive  crossing}
\label{fig:positivecrossing}
\end{subfigure}
\begin{subfigure}{0.32\linewidth}
\[ \begin{array}{c} \scalebox{0.5}{\fontsize{204pt}{20pt}\selectfont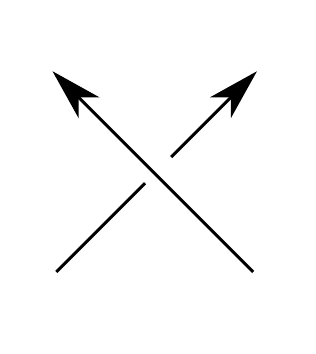} \end{array} \]
\caption{Negative  crossing}
\label{fig:negativecrossing}
\end{subfigure}
\begin{subfigure}{0.32\linewidth}
\[ \begin{array}{c} \scalebox{0.5}{\fontsize{204pt}{20pt}\selectfont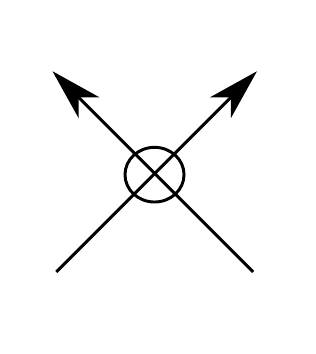} \end{array} \]
\caption{Virtual crossing}
\label{fig:virtualecrossing}
\end{subfigure}
\caption{Labeled crossings}
\label{fig:labeledcrossings}
\end{figure}

For positive even crossings, as shown in figure \ref{fig:positivecrossing}, the relationships are
\begin{align}\label{eqn:positiveevenrelations}
z&=xysx^{-1}s^{-1}, &
w&=sxs^{-1}.
\end{align}

For negative even crossings, as shown in figure \ref{fig:negativecrossing}, the relationships are
\begin{align}\label{eqn:negativeevenrelations}
z&=s^{-1}ys, &
w&=s^{-1}y^{-1} sxy.
\end{align}

The relationship associated to odd crossings are independent of the crossing type. Using the labels in
figure \ref{fig:labeledcrossings}, the odd crossings have the relationship
\begin{align}\label{eqn:oddrelations}
z&= \theta^{-1} y \theta, &
w &= \theta x \theta^{-1}.
\end{align}

The virtual crossings have  the relations
\begin{align}\label{eqn:virtualrelations}
z&= q^{-1} y q, &
w&= q x q^{-1}.
\end{align}

\begin{theorem} For all virtual knot diagrams $K$ and $\hat{K}$, if $K$ is related to $\hat{K}$ by a sequence of diagrammatic moves then  $PG_{K} $ is isomorphic to $PG_{\hat{K}}$.
\end{theorem}
\begin{proof}
We show invariance under  diagrammatic that include  odd crossings.  Invariance under moves involving only even crossings or even and virtual crossings are shown in BDGGHN \cite{boden}. The case of a Reidemeister III move involving even and odd crossings is analogous to the case of a virtual Reidemeister IV move that includes an even crossing.

\begin{figure}
\[ \begin{array}{c} \scalebox{0.5}{\fontsize{204pt}{20pt}\selectfont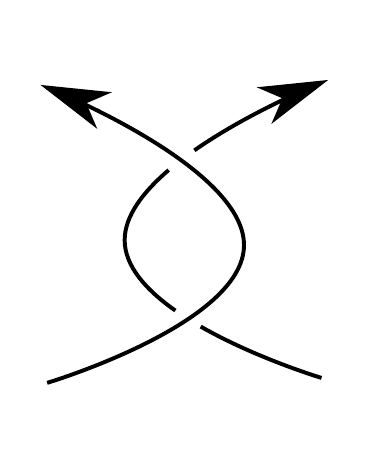} \end{array} \leftrightarrow
\begin{array}{c} \scalebox{0.5}{\fontsize{204pt}{20pt}\selectfont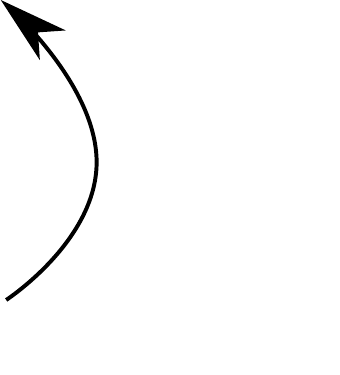} \end{array} \]
\caption{Labeled Reidemeister II}
\label{fig:r2movelabeled}
\end{figure}

We begin by examining a Reidemeister II move as shown in
figure \ref{fig:r2movelabeled}. In a Reidemeister II move, both crossings are either even or odd.
We prove that  $a=e$ and $b=f$. For even crossings, we use equations \ref{eqn:positiveevenrelations} and \ref{eqn:negativeevenrelations}. The diagram on the left hand side of the figure determines the relations
\begin{align*}
c&=absa^{-1}s^{-1}, & d &=sas^{-1}, \\
e&=s^{-1}d s, & f &=s^{-1}d^{-1}s c d.
\end{align*}
Reducing these relations, we see that $a=e$ and $b=f$.
If both crossings are odd, we use equations \ref{eqn:oddrelations} and the relations from the left hand side of the figure are
\begin{align*}
c&= \theta^{-1} b \theta,  & d&=\theta a \theta^{-1}, \\
e&= \theta^{-1} d \theta, & f &= \theta c \theta^{-1}.
\end{align*}
Again, $a=e$ and $b=f$ after rewriting. The virtual Reidemeister II case is analogous to the odd classical case.

\begin{figure}
\[ \begin{array}{c} \scalebox{0.5}{\fontsize{204pt}{20pt}\selectfont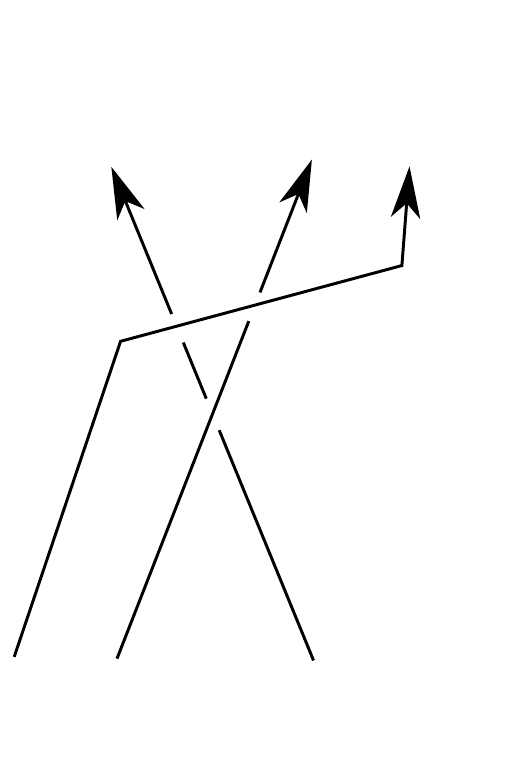} \end{array} \leftrightarrow
\begin{array}{c} \scalebox{0.5}{\fontsize{204pt}{20pt}\selectfont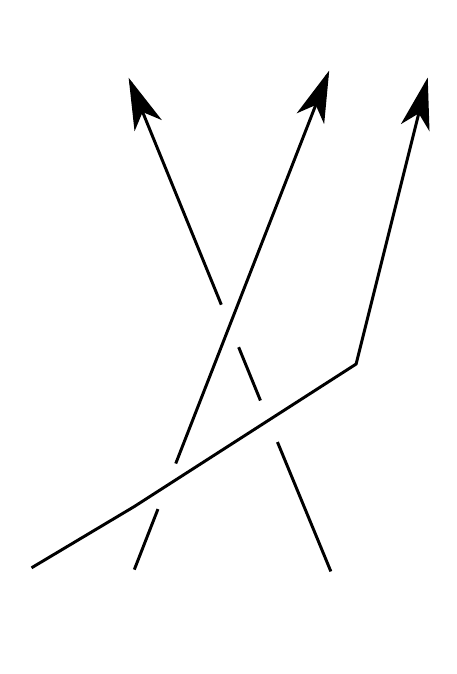} \end{array} \]
\caption{Labeled Reidemeister III}
\label{fig:r3movelabeled}
\end{figure}

We now consider the Reidemeister III move. Note that in a Reidemeister III move involving both even and odd crossings, the move must contain two odd crossings.
The odd crossing relationship is independent of crossing sign, so we only need to consider the case shown in figure \ref{fig:r3movelabeled}. We assume that the two crossings in the over passing strand are odd crossings. Using equations  \ref{eqn:positiveevenrelations}, and \ref{eqn:oddrelations}, we obtain from the right hand side
\begin{align*}
a&= yzsy^{-1} s^{-1}, & b &= sys^{-1}, \\
f&= \theta x \theta^{-1}, & c&= \theta^{-1} a \theta, \\
e&= \theta f \theta^{-1}, & d &= \theta^{-1} b \theta.
\end{align*}
Consequently,
\begin{align*}
e& = \theta^2 x \theta^{-2}, & d&= \theta^{-1} s y s^{-1} \theta^{-1},
& c&= \theta^{-1} y z s y^{-1} s^{-1} \theta .
\end{align*}

From the left hand side,
we obtain
\begin{align*}
f&= \theta x \theta^{-1}, & a&= \theta^{-1} y \theta, \\
b&= \theta^{-1} z \theta, &  e&= \theta f \theta^{-1}, \\
c&=absa^{-1} s^{-1}, & d&=sas^{-1}.
\end{align*}
We reduce these relations to obtain
\begin{align*}
e&= \theta^2 x \theta^{-2},
&  d &= s \theta^{-1} y \theta s^{-1},
& c&= \theta^{-1} y z s y^{-1} \theta s^{-1}.
\end{align*}
The two diagrams produce equivalent relations after adding the commutator $[s, \theta] = 1$ . The other two cases where the even crossing is included in the overpassing strand follows similarly.
The commutators $[s, q]=1$ and $[\theta, q] = 1 $ are added to ensure invariance under the virtual Reidemeister III move  (with either even or odd crossings). The defined group is an invariant of a virtual knot.
\end{proof}

\section{The Virtual Parity Alexander Module}

In this section, we construct the virtual parity Alexander module.
From a knot diagram with $n$ crossings, we obtain $2n+3$ generators and $2n+3$ relations.
The group $PG_K$ is a quotient of $F_{2n+3}$:
\begin{equation*}
PG_K = \lbrace a_1, a_2 \ldots a_{2n}, s,q, \theta \vert  r_1, r_2, \ldots r_{2n+3} \rbrace
\end{equation*}
where
\begin{align} \label{eqn:commutators}
r_{2n+1}&= [s,q], & r_{2n+2} &=[s, \theta],  & r_{2n+3}&= [\theta, q].
\end{align}

Two equivalent knot diagrams produce different but isomorphic presentations of the same group. One presentation can be transformed into the other by following the sequence of Reidemeister moves relating the two diagrams. However, since it is known that such a sequence of exchanges exists then the Tietze Transformation theorem \cite{Tietze} can also be applied. The Tietze theorem specifies that the transformation can be achieved through exactly two types of transformations: 1) a consequence of existing relations and 2) the introduction of a new generator $x$ which is equated with an existing word $w$ to form the relation $x^{-1}w =1 $.

We review Fox's free differentials. The differentials
 linearize the relations in a multiplicative group to produce a system of homogeneous linear equations.
The free group on the $m$ elements is denoted $F_m$ and the group ring is denoted as $ \mathbb{Z} [F_{m}]$.
Fox's free differentials are a set of maps $\frac{ \partial }{\partial x_j} 	: \mathbb{Z} [ F_n] \rightarrow \mathbb{Z} [F_n]$ with the following properties:
\begin{equation}
\frac{ \partial x_i}{\partial x_j} = \begin{cases}
											1 & i= j \\
											0 & i \neq j
											\end{cases} 					
\end{equation}
and
\begin{equation}
\frac{ \partial aw}{\partial x_j} 	= \frac{ \partial a}{\partial x_j} 	+ a \frac{ \partial w}{\partial x_j} .
\end{equation}

Fox's Fundamental Identity establishes a relationship between an element $w$ of $F_m$ and the differentials:
\begin{equation} \label{eqn:foxid}
w-1 = \sum \frac{ \partial w }{\partial a_j} (a_j -1).
\end{equation}

The  canonical homomorphism from $F_{2n+3} $ to $PG_K$ induces a homomorphism from $ \mathbb{Z} [F_{2n+3}] $ to $ \mathbb{Z}[PG_K]$. The free differential maps are applied to the words $r_{ij}$ determined by the crossings in the diagram of $K$. The induced map  from $ \mathbb{Z} [F_{2n+3}] $ to $ \mathbb{Z}[PG_K]$ and equation \ref{eqn:foxid}
results in  a system of  linear equations
\begin{equation} \label{eqn:appoffox}
 \frac{ \partial r_{i}}{ \partial \theta} (\theta -1) + \frac{ \partial r_{i}}{ \partial q} (q -1)   +  \frac{ \partial r_{i}}{ \partial s} (s -1) + \sum_{j=1} ^{2n} \frac{ \partial r_{i}}{ \partial a_j} (a_j -1)  = r_{i} -1 .
\end{equation}

In  $\mathbb{Z}[PG_K]$, $r_{i} -1 =0 $ so that equation \ref{eqn:appoffox} becomes
\begin{equation*}
 \frac{ \partial r_{i}}{ \partial \theta} (\theta -1) + \frac{ \partial r_{i}}{ \partial q} (q -1)   +  \frac{ \partial r_{i}}{ \partial s} (s -1) + \sum_{j=1} ^{2n} \frac{ \partial r_{i}}{ \partial a_j} (a_j -1)  = 0.
\end{equation*}

Now, mapping  $F_{2n+3}$ into the group ring $\mathbb{Z}[t,q,s,\theta]$ by sending $a_i$ to $t$ for $1 \leq i \leq 2n$ and the remaining generators to $q,s,$ and $\theta$ induces a corresponding quotient of $ \mathbb{Z}[PG_k]$.

From the group presentation, we obtain a system of homogeneous linear equations with coefficients in
$ \mathbb{Z} [t,q,s,\theta]$. If two groups are isomorphic, then the two systems of linear equations have equivalent solution sets.

Let $M$ denote the matrix corresponding to the system of equations.
Observe that $M \bar{x} = \bar{0}$ has the trivial solution $ \bar{0}$ and the solution $ \bar{1} $. This implies that the determinant of $M$ is zero. Additional solutions to the system of equations are found in quotients of $ \mathbb{Z}[t,q,s, \theta]$; solutions are multiples of the greatest common divisor (gcd) of the minors of $M$.

Denote the gcd of the minors as $ \Delta_1 (K)$ and additional solutions are in
$  \mathbb{Z}[t,q,s, \theta] / \Delta_1 (K)$. Let $\Delta_{r} (K)$ denote the gcd of the minors of rank $n-r$;  $ \Delta_r (K)$ is a divisor of $ \Delta_{r-1} (K)$ via the definition of determinant using a co-factor expansion. This produces a sequence of ascending ideals associated to the knot $K$:
\begin{equation}
(\Delta_1 (K) ) \subset (\Delta_2 (K) ) \ldots ( \Delta_n (K)) \subset (1).
\end{equation}

The determinants of row equivalent matrices have a well understood relationship and  two row equivalent matrices produce determinants that will differ only by scalars in $ \mathbb{Z} [s,t,q, \theta]$.

We construct the $(2n+3) \times (2n+3) $ matrix $M$ for $PG_K$.
The relation $ r_{2n+1}$ is $[s,q] = 1$, we obtain the linearized relations
\begin{gather*}
\frac{\partial r_{2n+1}}{\partial a_i} = 0 \text{ for } i \in 1,2, \ldots 2n, \\
\frac{\partial r_{2n+1}}{\partial s} = 1-q,  \\
\frac{\partial r_{2n+1}}{\partial q} = s-1, \\
\frac{\partial r_{2n+1}}{\partial \theta} = 0.
\end{gather*}

The relations are computed similarly for $r_{2n+2}$ and $r_{2n+3}$ from equation \ref{eqn:commutators}, so that

\begin{equation}
M = \begin{bmatrix} A & ( \frac{\partial r_i }{ \partial s} ) & ( \frac{\partial r_i }{ \partial q} ) &
( \frac{\partial r_i }{ \partial \theta}  ) \\
0 & 1-q & s-1 & 0 \\
0 & 1- \theta & 0 & s-1 \\
0 & 0 & \theta -1 & 1-q \end{bmatrix}
\end{equation}
where the entries in $A$ have the form $ \frac{ \partial r_i }{ \partial a_j } $ and
$ ( \frac{ \partial r_i }{ \partial * }) $ is column vector of length $2n$. (Recall that two matrices obtained from groups related by the Tietze transformation theorem are related by row reduction and (possibly) the introduction or removal of a dependent column.)
The determinant of $M$ is zero by application of Fox's Fundamental Identity.

The \textit{parity virtual Alexander polynomial}, denoted $\Phi \Delta_K (s,t,q,\theta) $, is defined to be the greatest common divisor of the minors of $M$.
\begin{theorem} The parity virtual Alexander polynomial, $\Phi \Delta_K (s,t,q,\theta) $, is an invariant of the virtual knot $K$. \end{theorem}

\begin{proof} This follows from previous work in the section. \end{proof}

\begin{proposition} The invariant $ \Phi \Delta_K (s,t,q, \theta) = det(A)$. \end{proposition}
\begin{proof}
The matrix $M$ has the form:
\begin{equation}
\begin{bmatrix} A & \frac{\partial r_i}{ \partial s} & \frac{ \partial r_i}{\partial q} & \frac{\partial r_i}{\partial \theta} \\
\bar{0} & 1-q & s-1 & 0 \\
\bar{0} & 1- \theta & 0 & s-1 \\
\bar{0} & 0 & \theta -1 & 1-q \end{bmatrix}.
\end{equation}
The submatrix $ A= ( \frac{\partial r_i}{\partial x_j} ) $ is a $ 2n \times 2n $ matrix
and the entries $ \frac{ \partial r_i}{\partial \star } $ are column vectors.
The determinant of the submatrix
\begin{equation*}
\begin{bmatrix} 1-q & s-1 & 0 \\
 1- \theta & 0 & s-1 \\
 0 & \theta -1 & 1-q \end{bmatrix}
\end{equation*}
is  zero. Hence, the determinant of $M$ is zero.
We apply Fox's fundamental identity
\begin{equation} \label{eqn:fox2}
\sum A_j (t-1) + \frac{ \partial r}{\partial s} (s-1) + \frac{ \partial r}{\partial q} (q-1) +
\frac{ \partial r}{\partial \theta} ( \theta -1) =0
\end{equation}
where $A_j$ is the $j$th column vector in $A$ and $\frac{ \partial r_i}{ \partial \star }$ is a column vector to compute the minors of $M$.

We use the notation $M_{k,l}$ to indicate the matrix obtained from $M$ by deleting the $k$th row and the $j$th column.
Delete the last row and column to produce
\begin{equation*}
 M_{2n+3, 2n+3} = \begin{bmatrix} A & \frac{\partial r_i}{ \partial s} & \frac{ \partial r_i}{\partial q}  \\
\bar{0} & 1-q & s-1 \\
\bar{0} & 1- \theta & 0  \\
\end{bmatrix}.
\end{equation*}
Computation shows that $Det(M_{2n+3, 2n+3}) = \pm 1 (1-\theta)(s-1)Det(A)$.
Similarly, $Det(M_{2n+3,  2n+2}) = \pm 1 (s-1) (1-q) Det(A)$ and $Det(M_{2n+3, 2n+1}) = \pm 1 (s-1)^2 Det(A)$.
Now, let $\hat{A}_k$ denote the matrix $A$ with the $k$th column deleted.
Then
\begin{equation} \label{eqn:detcalc}
Det( M_{2n+3, k} )  =
 \begin{vmatrix} \hat{A}_k & \frac{\partial r_i}{ \partial s} & \frac{ \partial r_i}{\partial q} & \frac{\partial r_i}{\partial \theta} \\
\bar{0} & 1-q & s-1 & 0 \\
\bar{0} & 1- \theta & 0 & s-1
 \end{vmatrix} .
\end{equation}
Next, expand and simplify equation \ref{eqn:detcalc}  using equation \ref{eqn:fox2}
\begin{align*}
Det(M_{2n+3, k}) &=  (1- \theta)
 \begin{vmatrix} \hat{A}_k & \frac{\partial r_i}{ \partial q} & \frac{ \partial r_i}{\partial \theta}   \\
\bar{0} & s-1 & 0
 \end{vmatrix}
  +(s-1)
\begin{vmatrix} \hat{A}_k & \frac{\partial r_i}{ \partial s} & \frac{ \partial r_i}{\partial q} &  \\
\bar{0} & 1-q & s-1
\end{vmatrix} \\
&=
 (1- \theta) (s-1) \begin{vmatrix} \hat{A}_k& \frac{ \partial r_i}{\partial \theta} \end{vmatrix}
 + (s-1)(1-q) \begin{vmatrix} \hat{A}_k  & \frac{ \partial r_i}{\partial q} \end{vmatrix}
 + (s-1)^2 \begin{vmatrix} \hat{A}_k  & \frac{ \partial r_i}{\partial s} \end{vmatrix} \\
 &= \begin{vmatrix} \hat{A}_k  &  (1- \theta )(s-1) \frac{ \partial r_i}{\partial \theta}
 + (s-1)(1-q)\frac{ \partial r_i}{\partial q} + (s-1)^2 \frac{ \partial r_i}{\partial s} \end{vmatrix} \\
 &= \begin{vmatrix} \hat{A}_k    &  \sum_{j=1} ^m A_j (1-t)(s-1) \end{vmatrix} \\
 &= \pm 1 Det(A) (1-t)(s-1).
\end{align*}
Based on these computations, deleting any of the first $2n$ columns and one of the last three rows results in a minor with $Det(A)$ as a factor. Deleting any other row and column results in a minor with value zero. The greatest common divisor of the minors of $M$ is $Det(A)$.
 \end{proof}

To simplify notation, we will denote $\Phi \Delta_K (s,t,q, \theta)$ as $ \Phi \Delta (K) $ when we do not need to directly reference the variables.
\begin{example}
We compute the invariant for the the knots shown in figure \ref{fig:virtualknotexamples}. The knot diagrams are listed in Jeremy Green's knot tables \cite{GreenVKTables}.
\begin{figure}
\begin{subfigure}{0.45\linewidth}
\[ \begin{array}{c} \scalebox{0.5}{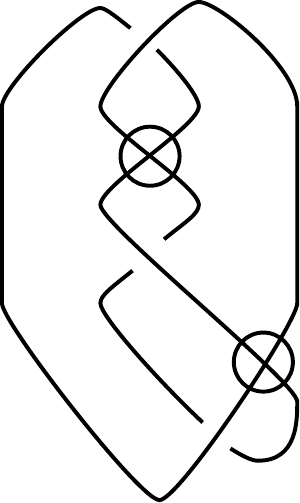} \end{array} \]
\caption{Knot 3.1}
\label{fig:knot3p1}
\end{subfigure}
\begin{subfigure}{0.45\linewidth}
\[ \begin{array}{c} \scalebox{0.5}{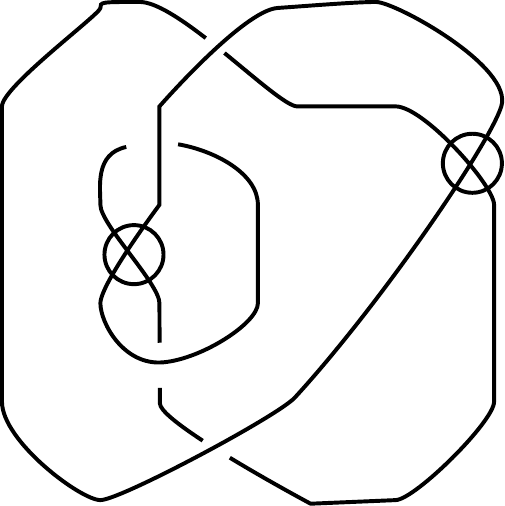} \end{array} \]
\caption{Knot 4.7}
\label{fig:knot4p7}
\end{subfigure} \\
\begin{subfigure}{0.45\linewidth}
\[ \begin{array}{c} \scalebox{0.5} {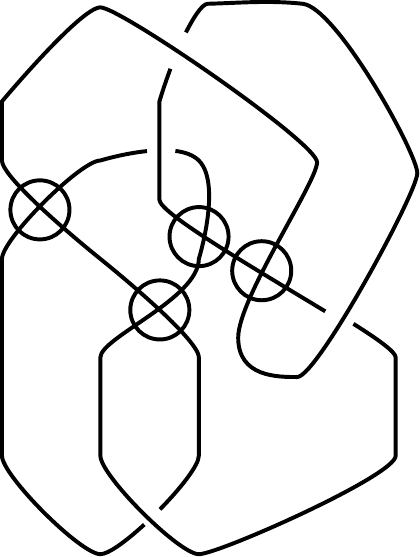} \end{array} \]
\caption{Knot 4.9}
\label{fig:knot4p9}
\end{subfigure}
\begin{subfigure}{0.45\linewidth}
\[ \begin{array}{c} \scalebox{0.5}{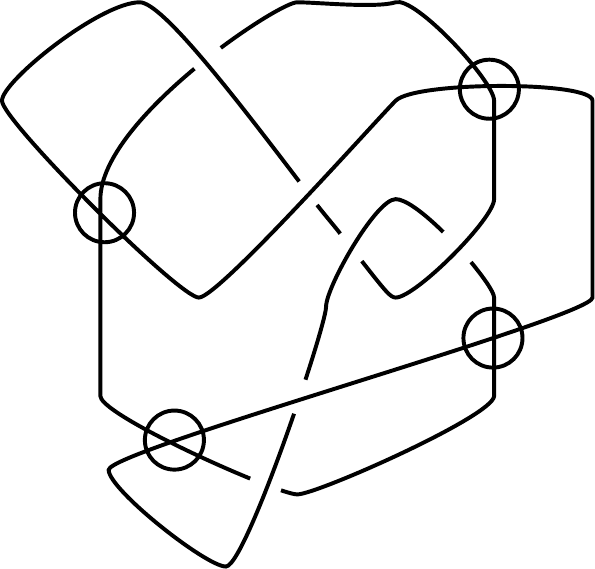} \end{array} \]
\caption{Knot 6.32008}
\label{fig:knot6p32008}
\end{subfigure}
\caption{Virtual knot examples}
\label{fig:virtualknotexamples}
\end{figure}
\begin{align}
\label{eqn:k3p1}
\Phi \Delta(K_{3.1})&=\frac{1}{q} + \frac{q}{st} - \frac{\theta^2}{q} + \frac{\theta^2}{stq} \\
\label{eqn:k4p7}
\Phi \Delta(K_{4.7})&=\frac{1}{q} -q  \\
\label{eqn:k4p9}
\Phi \Delta(K_{4.9})&= -1 + \frac{1}{s^2t^2}	+ \frac{1}{sq} - \frac{1}{s^2 t q}	-\frac{q}{st^2} +\frac{q}{t}   \\
\label{eqn:k6p3}
\Phi \Delta (K_{6.32008})&= 1-\frac{1}{st} + \frac{1}{sq} - \frac{t}{q} -sq + \frac{q}{t} - \frac{q}{\theta} + \frac{stq}{\theta}
\end{align}

\end{example}

We see that the virtual parity Alexander polynomial does not vanish on the knot 6.32008, whereas the parity Alexander polynomial (Kaestner and Kauffman \cite{KaestnerKauffman}) does vanish.

\section{Properties}
We study the properties of the invariant.

\subsection{Lower bounds on crossing numbers}

The virtual Alexander polynomial determines lower bounds on both the number of virtual and odd crossings in the diagram. We follow the methods and definitions of BDGGHN \cite{boden}.  Define the $x-width$ of a polynomial $f(x)$ to be the maximum degree of $x$ in the polynomial minus the minimum degree.
\begin{theorem}
For a virtual knot $K$, the virtual parity Alexander polynomial determines a lower bound on the number of virtual and odd crossings. Let $v$ (respectively $o$) denote the minimum number of virtual  (respectively odd) crossings in any diagram of $K$. Then
\begin{align}
\theta - width \Phi \Delta (K) & \leq 2v \\
q-width \Phi \Delta (K) &  \leq 2o \\
\end{align}
\end{theorem}
\begin{proof}
There are two equations associated to an individual crossing. In the construction, each equation contributes a  row contributes either $q$ or $q^{-1}$ (respectively
$\theta $ or $ \theta^{-1}$ ) to the determinant.
\end{proof}

We apply these bounds to the examples.

\begin{example}

The diagram of $K_{3.1}$ contains two odd crossings and two virtual crossings. From the polynomial
in equation \ref{eqn:k3p1},
\begin{align*}
q- \it{width} \Phi \Delta (K_{3.1}) &= 2, &
\theta - \it{width} \Phi \Delta(K_{3.1}) &=2.
\end{align*}
The polynomial determines a lower bound of $1$ for both the odd and virtual crossings.

The diagram of $K_{4.7} $ contains four odd crossings and two virtual crossings; from equation \ref{eqn:k4p7}
\begin{align*}
q- \it{width} \Phi \Delta (K_{4.7}) &= 2, &
\theta - \it{width} \Phi \Delta (K_{4.7}) &=0.
\end{align*}
This gives a lower bound of $1$ on the virtual crossings and $0$ on the odd crossings.

The diagram of $K_{4.9} $ contains two odd crossings and four virtual crossings; from equation \ref{eqn:k4p9}
\begin{align*}
q- \it{width} \Phi \Delta (K_{4.9}) &= 2, &
\theta - \it{width} \Phi \Delta (K_{4.9}) &=0.
\end{align*}
This gives a lower bound of $1$ on the virtual crossings and $0$ on the odd crossings.

The diagram of $K_{6.32008} $ contains three odd crossings and four virtual crossings;
from equation \ref{eqn:k6p3}
\begin{align*}
q- \it{width} \Phi \Delta (6.32008) &= 2, &
\theta - \it{width} \Phi \Delta (K_{6.32008}) &=1.
\end{align*}
This gives a lower bound of $1$ on the virtual crossings and $1$ on the odd crossings.

\end{example}

\subsection{Skein relations}

Let $K_+$ denote a knot diagram with a positive crossing and let $K_-$ denote the knot diagram obtained from $K_+$ by replacing the selected positive crossing with a negative crossing.
Let $K_v$ denote the knot diagram with the selected crossing replaced by a vertical smoothing.  We first consider even crossings.
\begin{theorem} For a knot diagram $K_+$, with a even, positive crossing, $\Phi \Delta (K_+) - \Phi \Delta (K_-) = (1-st) \Phi \Delta(K_v)$ \end{theorem}
\begin{proof} We construct the matrices so that entries corresponding to the selected crossing are in the first two rows of the matrix. The $*$ entries in the matrix for all three diagrams and sub-matrices obtained by deleting the first two rows are identical.  Let
$A_{ij}$ denote the matrix obtained by deleting the first two rows and columns $i$ and $j$.
The matrix obtained from $K_+$ is
\begin{equation*}
A = \begin{bmatrix} \begin{matrix}
0 & 1-st & t & -1  \\
-1 & s & 0 & 0 \\ * & * & * & *   \end{matrix}  &
\begin{matrix} * \\  * \\
  *  \end{matrix}
\end{bmatrix}
\end{equation*}
where the relation for $z$ is in the first row and the expression for $w$ is in the second.
Then $\Phi \Delta(K_+) = (1-st)A_{12} - t A_{13} - A_{14} - st A_{23} - s A_{24}$.
For $K_-$, the matrix obtained is
\begin{equation*}
\begin{bmatrix}
\begin{matrix} 0 & 0 & \frac{1}{s}  & -1  \\
-1 & \frac{1}{t} & ( - \frac{1}{st} +1)  & 0 \\ * & * & * & * \end{matrix} &
\begin{matrix} *  \\
* \\ * \end{matrix}
\end{bmatrix}
\end{equation*}
with $ \Phi \Delta (K_-)= - \frac{1}{s} A_{13}-  A_{14} - \frac{1}{st} A_{23}  - \frac{1}{t} A_{24} + (1-\frac{1}{st}) A_{34}$.
Note that the expressions describing $w$ and $z$ are assigned to the same rows.

In $K_v$, we replace the equations by equating the labels on the strands and obtain
From the matrix
\begin{equation*}
\begin{bmatrix} \begin{matrix}
0 & 1 & 0  & -1  \\
-1 & 0 & 1  & 0 \\
* & * & * & *  \end{matrix} &
\begin{matrix} * \\ * \\ * \end{matrix}
\end{bmatrix}
\end{equation*}
Now $ \Phi \Delta (K_v) = A_{12} - A_{14} + A_{23} - A_{34} $.
We observe that
$ \Phi \Delta (K_+)- st \Phi \Delta (K_-)=(1-st )(A_{12} -   A_{14} +  A_{23} + A_{34} )$.
\end{proof}

We now consider odd crossing.

\begin{theorem} For an odd crossing, $\Phi \Delta (K_+) - \Phi \Delta (K_-) =0$. \end{theorem}
\begin{proof}
The relationships determined by both diagrams are identical.
\end{proof}

For a virtual knot $K$, a \textit{degree one odd Vassiliev invariant}, denoted $ v(k)$, is a virtual knot invariant that satisfies the condition $v(K_+) - v(K_-)=0 $ when the selected crossing is odd and
$v(K_+) - v(K_-) \neq 0 $ when the crossing is even. The virtual parity Alexander polynomial is an odd degree one Vassiliev invariant.

\begin{corollary} For a non-classical knot diagram $K$ with $ \Phi \Delta (K) \neq 0$,  if the knot diagram $K'$ obtained by switching the odd crossings from positive to negative (or vice versa) then $K'$ is non-trivial.  \end{corollary}

As a result, for any $K$ with $ \Phi \Delta (K) \neq 0$, $K$ can not be unknotted by changing the sign of only odd crossings. Further, many $K$ can not be changed into a classical knot by changing the sign of odd crossings.
Recall that the Alexander biquandle polynomial is zero for all classical knots. The index of a crossing, see Chrisman and Dye \cite{chrismandye} and Kauffman and Folwaczny \cite{kauffmanfolwaczny}, can show that changing the sign of specific odd crossings will not unknot or classicalize the knot diagram. One question to consider is if the two invariants detect the same set of unknottable diagrams.

\begin{example} The diagram of $K_{3.1}$ contains two odd crossings and two virtual crossings.
The sign of either of the two crossings can be switched without changing the value of the polynomial.

The diagram of $K_{4.7} $ contains four odd crossings and two virtual crossings;
This diagram cannot be turned into a classical diagram by changing the sign of the odd crossings.

The diagram of $K_{4.9} $ contains two odd crossings and four virtual crossings; the diagram can not be changed into a classical diagram by changing the sign of the odd crossings.

The diagram of $K_{6.32008} $ contains four odd crossings and two virtual crossings;
the diagram can not be turned into a classical diagram by changing the sign of the odd crossings.
\end{example}

\subsection{Symmetries}

We consider the effect of various symmetries on the polynomial.

\subsubsection{Action under reverse}

Given a virtual knot $K$, the \textit{reverse} of $K$, denoted $K'$, is obtained by reversing the orientation of the virtual knot. (More generally, for a virtual link, the inverse is obtained by reversing the orientation on all components of the link \cite{CransHenrichNelson}.) See figure \ref{fig:reversecrossing}.
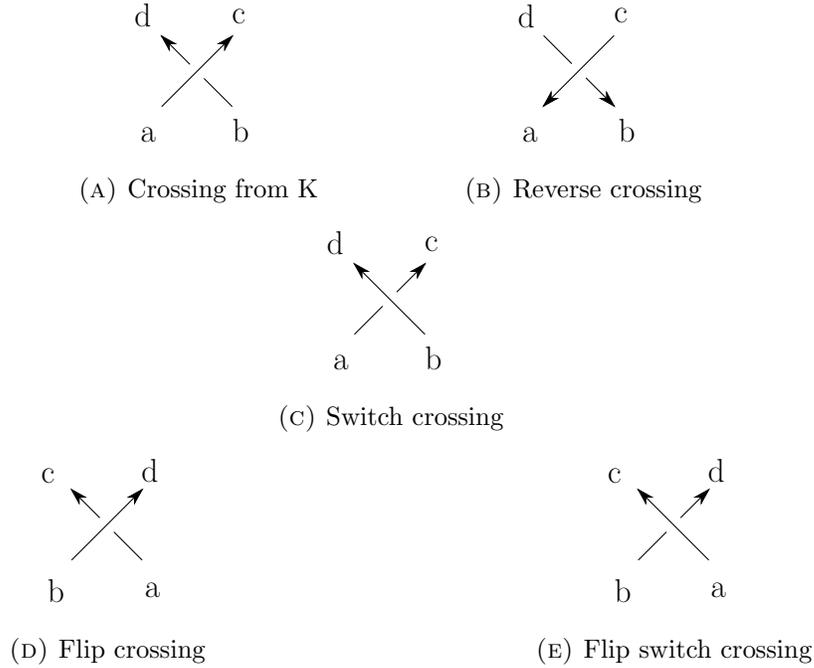
\begin{figure}
\begin{subfigure}{0.33\linewidth}
\[ \begin{array}{c} \scalebox{0.5}{\fontsize{204pt}{20pt}\selectfont%% Creator: Inkscape inkscape 0.92.4, www.inkscape.org
%% Accompanies image file 'standardcrossing.pdf' (pdf, eps, ps)
\begingroup%
  \makeatletter%
  \providecommand\color[2][]{%
    \errmessage{(Inkscape) Color is used for the text in Inkscape, but the package 'color.sty' is not loaded}%
    \renewcommand\color[2][]{}%
  }%
  \providecommand\transparent[1]{%
    \errmessage{(Inkscape) Transparency is used (non-zero) for the text in Inkscape, but the package 'transparent.sty' is not loaded}%
    \renewcommand\transparent[1]{}%
  }%
  \providecommand\rotatebox[2]{#2}%
  \newcommand*\fsize{\dimexpr\f@size pt\relax}%
  \newcommand*\lineheight[1]{\fontsize{\fsize}{#1\fsize}\selectfont}%
  \ifx\svgwidth\undefined%
    \setlength{\unitlength}{93.75990452bp}%
    \ifx\svgscale\undefined%
      \relax%
    \else%
      \setlength{\unitlength}{\unitlength * \real{\svgscale}}%
    \fi%
  \else%
    \setlength{\unitlength}{\svgwidth}%
  \fi%
  \global\let\svgwidth\undefined%
  \global\let\svgscale\undefined%
  \makeatother%
  \begin{picture}(1,1.27563896)%
    \lineheight{1}%
    \setlength\tabcolsep{0pt}%
    \put(0,0){\includegraphics[width=\unitlength,page=1]{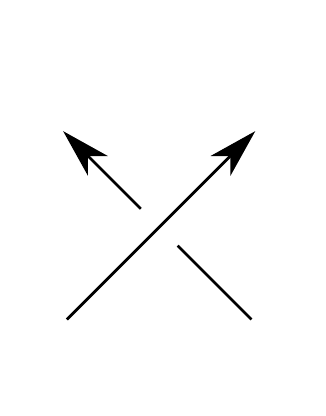}}%
    \put(0.77241269,0.01081333){\color[rgb]{0,0,0}\makebox(0,0)[lt]{\lineheight{0}\smash{\begin{tabular}[t]{l}b\end{tabular}}}}%
    \put(0.03548267,0.01081333){\color[rgb]{0,0,0}\makebox(0,0)[lt]{\lineheight{0}\smash{\begin{tabular}[t]{l}a\end{tabular}}}}%
    \put(0.77197998,0.96410617){\color[rgb]{0,0,0}\makebox(0,0)[lt]{\lineheight{0}\smash{\begin{tabular}[t]{l}c\end{tabular}}}}%
    \put(-0.0116843,0.93814138){\color[rgb]{0,0,0}\makebox(0,0)[lt]{\lineheight{0}\smash{\begin{tabular}[t]{l}d\end{tabular}}}}%
  \end{picture}%
\endgroup%
} \end{array}  \]
\caption{Crossing from K}
\label{fig:standardcrossing}
\end{subfigure}
\begin{subfigure}{0.33\linewidth}
\[ \begin{array}{c} \scalebox{0.5}{\fontsize{204pt}{20pt}\selectfont%% Creator: Inkscape inkscape 0.92.4, www.inkscape.org
%% Accompanies image file 'kreversecrossing.pdf' (pdf, eps, ps)
\begingroup%
  \makeatletter%
  \providecommand\color[2][]{%
    \errmessage{(Inkscape) Color is used for the text in Inkscape, but the package 'color.sty' is not loaded}%
    \renewcommand\color[2][]{}%
  }%
  \providecommand\transparent[1]{%
    \errmessage{(Inkscape) Transparency is used (non-zero) for the text in Inkscape, but the package 'transparent.sty' is not loaded}%
    \renewcommand\transparent[1]{}%
  }%
  \providecommand\rotatebox[2]{#2}%
  \newcommand*\fsize{\dimexpr\f@size pt\relax}%
  \newcommand*\lineheight[1]{\fontsize{\fsize}{#1\fsize}\selectfont}%
  \ifx\svgwidth\undefined%
    \setlength{\unitlength}{96.19424109bp}%
    \ifx\svgscale\undefined%
      \relax%
    \else%
      \setlength{\unitlength}{\unitlength * \real{\svgscale}}%
    \fi%
  \else%
    \setlength{\unitlength}{\svgwidth}%
  \fi%
  \global\let\svgwidth\undefined%
  \global\let\svgscale\undefined%
  \makeatother%
  \begin{picture}(1,1.23323497)%
    \lineheight{1}%
    \setlength\tabcolsep{0pt}%
    \put(0,0){\includegraphics[width=\unitlength,page=1]{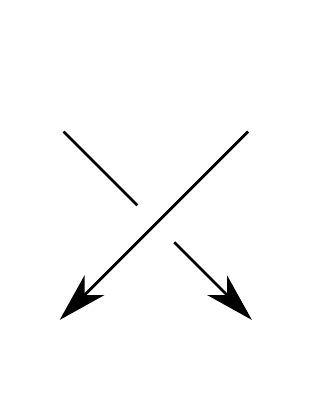}}%
    \put(0.77817212,0.01053968){\color[rgb]{0,0,0}\makebox(0,0)[lt]{\lineheight{0}\smash{\begin{tabular}[t]{l}b\end{tabular}}}}%
    \put(0.02446216,0.01053968){\color[rgb]{0,0,0}\makebox(0,0)[lt]{\lineheight{0}\smash{\begin{tabular}[t]{l}a\end{tabular}}}}%
    \put(0.74232133,0.93970804){\color[rgb]{0,0,0}\makebox(0,0)[lt]{\lineheight{0}\smash{\begin{tabular}[t]{l}c\end{tabular}}}}%
    \put(-0.01138861,0.90427826){\color[rgb]{0,0,0}\makebox(0,0)[lt]{\lineheight{0}\smash{\begin{tabular}[t]{l}d\end{tabular}}}}%
  \end{picture}%
\endgroup%
} \end{array}  \]
\caption{Reverse crossing}
\label{fig:reversecrossing}
\end{subfigure}
\begin{subfigure}{0.33\linewidth}
\[ \begin{array}{c} \scalebox{0.5}{\fontsize{204pt}{20pt}\selectfont%% Creator: Inkscape inkscape 0.92.4, www.inkscape.org
%% Accompanies image file 'switchcrossing.pdf' (pdf, eps, ps)
\begingroup%
  \makeatletter%
  \providecommand\color[2][]{%
    \errmessage{(Inkscape) Color is used for the text in Inkscape, but the package 'color.sty' is not loaded}%
    \renewcommand\color[2][]{}%
  }%
  \providecommand\transparent[1]{%
    \errmessage{(Inkscape) Transparency is used (non-zero) for the text in Inkscape, but the package 'transparent.sty' is not loaded}%
    \renewcommand\transparent[1]{}%
  }%
  \providecommand\rotatebox[2]{#2}%
  \newcommand*\fsize{\dimexpr\f@size pt\relax}%
  \newcommand*\lineheight[1]{\fontsize{\fsize}{#1\fsize}\selectfont}%
  \ifx\svgwidth\undefined%
    \setlength{\unitlength}{94.24676981bp}%
    \ifx\svgscale\undefined%
      \relax%
    \else%
      \setlength{\unitlength}{\unitlength * \real{\svgscale}}%
    \fi%
  \else%
    \setlength{\unitlength}{\svgwidth}%
  \fi%
  \global\let\svgwidth\undefined%
  \global\let\svgscale\undefined%
  \makeatother%
  \begin{picture}(1,1.2690492)%
    \lineheight{1}%
    \setlength\tabcolsep{0pt}%
    \put(0,0){\includegraphics[width=\unitlength,page=1]{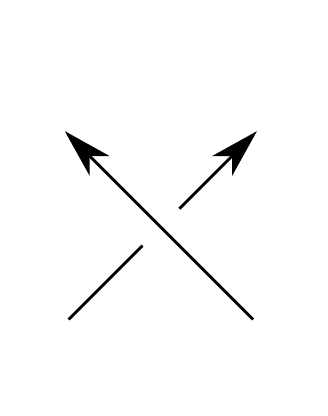}}%
    \put(0.77358837,0.01075747){\color[rgb]{0,0,0}\makebox(0,0)[lt]{\lineheight{0}\smash{\begin{tabular}[t]{l}b\end{tabular}}}}%
    \put(0.04046523,0.01075747){\color[rgb]{0,0,0}\makebox(0,0)[lt]{\lineheight{0}\smash{\begin{tabular}[t]{l}a\end{tabular}}}}%
    \put(0.7731579,0.95912573){\color[rgb]{0,0,0}\makebox(0,0)[lt]{\lineheight{0}\smash{\begin{tabular}[t]{l}c\end{tabular}}}}%
    \put(-0.01162394,0.93329508){\color[rgb]{0,0,0}\makebox(0,0)[lt]{\lineheight{0}\smash{\begin{tabular}[t]{l}d\end{tabular}}}}%
  \end{picture}%
\endgroup%
} \end{array}  \]
\caption{Switch crossing}
\label{fig:switchcrossing}
\end{subfigure}  \\
\begin{subfigure}{0.49\linewidth}
\[ \begin{array}{c} \scalebox{0.5}{\fontsize{204pt}{20pt}\selectfont%% Creator: Inkscape inkscape 0.92.4, www.inkscape.org
%% Accompanies image file 'flipcrossing.pdf' (pdf, eps, ps)
\begingroup%
  \makeatletter%
  \providecommand\color[2][]{%
    \errmessage{(Inkscape) Color is used for the text in Inkscape, but the package 'color.sty' is not loaded}%
    \renewcommand\color[2][]{}%
  }%
  \providecommand\transparent[1]{%
    \errmessage{(Inkscape) Transparency is used (non-zero) for the text in Inkscape, but the package 'transparent.sty' is not loaded}%
    \renewcommand\transparent[1]{}%
  }%
  \providecommand\rotatebox[2]{#2}%
  \newcommand*\fsize{\dimexpr\f@size pt\relax}%
  \newcommand*\lineheight[1]{\fontsize{\fsize}{#1\fsize}\selectfont}%
  \ifx\svgwidth\undefined%
    \setlength{\unitlength}{97.37392267bp}%
    \ifx\svgscale\undefined%
      \relax%
    \else%
      \setlength{\unitlength}{\unitlength * \real{\svgscale}}%
    \fi%
  \else%
    \setlength{\unitlength}{\svgwidth}%
  \fi%
  \global\let\svgwidth\undefined%
  \global\let\svgscale\undefined%
  \makeatother%
  \begin{picture}(1,1.2632927)%
    \lineheight{1}%
    \setlength\tabcolsep{0pt}%
    \put(0,0){\includegraphics[width=\unitlength,page=1]{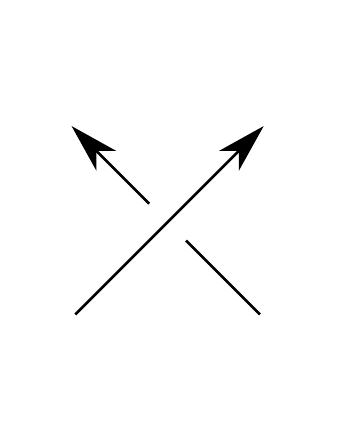}}%
    \put(0.03871441,0.01041199){\color[rgb]{0,0,0}\makebox(0,0)[lt]{\lineheight{0}\smash{\begin{tabular}[t]{l}b\end{tabular}}}}%
    \put(0.78912647,0.04541087){\color[rgb]{0,0,0}\makebox(0,0)[lt]{\lineheight{0}\smash{\begin{tabular}[t]{l}a\end{tabular}}}}%
    \put(-0.01170194,0.9333224){\color[rgb]{0,0,0}\makebox(0,0)[lt]{\lineheight{0}\smash{\begin{tabular}[t]{l}c\end{tabular}}}}%
    \put(0.76370997,0.93832129){\color[rgb]{0,0,0}\makebox(0,0)[lt]{\lineheight{0}\smash{\begin{tabular}[t]{l}d\end{tabular}}}}%
  \end{picture}%
\endgroup%
} \end{array}  \]
\caption{Flip crossing}
\label{fig:flipcrossing}
\end{subfigure}
\begin{subfigure}{0.49\linewidth}
\[ \begin{array}{c} \scalebox{0.5}{\fontsize{204pt}{20pt}\selectfont%% Creator: Inkscape inkscape 0.92.4, www.inkscape.org
%% Accompanies image file 'flipswitchcrossing.pdf' (pdf, eps, ps)
\begingroup%
  \makeatletter%
  \providecommand\color[2][]{%
    \errmessage{(Inkscape) Color is used for the text in Inkscape, but the package 'color.sty' is not loaded}%
    \renewcommand\color[2][]{}%
  }%
  \providecommand\transparent[1]{%
    \errmessage{(Inkscape) Transparency is used (non-zero) for the text in Inkscape, but the package 'transparent.sty' is not loaded}%
    \renewcommand\transparent[1]{}%
  }%
  \providecommand\rotatebox[2]{#2}%
  \newcommand*\fsize{\dimexpr\f@size pt\relax}%
  \newcommand*\lineheight[1]{\fontsize{\fsize}{#1\fsize}\selectfont}%
  \ifx\svgwidth\undefined%
    \setlength{\unitlength}{97.37392267bp}%
    \ifx\svgscale\undefined%
      \relax%
    \else%
      \setlength{\unitlength}{\unitlength * \real{\svgscale}}%
    \fi%
  \else%
    \setlength{\unitlength}{\svgwidth}%
  \fi%
  \global\let\svgwidth\undefined%
  \global\let\svgscale\undefined%
  \makeatother%
  \begin{picture}(1,1.2632927)%
    \lineheight{1}%
    \setlength\tabcolsep{0pt}%
    \put(0,0){\includegraphics[width=\unitlength,page=1]{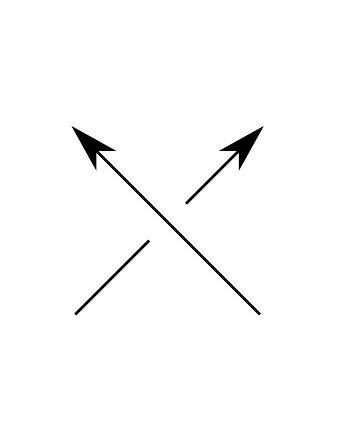}}%
    \put(0.03871441,0.01041199){\color[rgb]{0,0,0}\makebox(0,0)[lt]{\lineheight{0}\smash{\begin{tabular}[t]{l}b\end{tabular}}}}%
    \put(0.78412644,0.04541087){\color[rgb]{0,0,0}\makebox(0,0)[lt]{\lineheight{0}\smash{\begin{tabular}[t]{l}a\end{tabular}}}}%
    \put(-0.01170194,0.9333224){\color[rgb]{0,0,0}\makebox(0,0)[lt]{\lineheight{0}\smash{\begin{tabular}[t]{l}c\end{tabular}}}}%
    \put(0.76370997,0.93832129){\color[rgb]{0,0,0}\makebox(0,0)[lt]{\lineheight{0}\smash{\begin{tabular}[t]{l}d\end{tabular}}}}%
  \end{picture}%
\endgroup%
} \end{array}  \]
\caption{Flip switch crossing}
\label{fig:flipswitchcrossing}
\end{subfigure}
\caption{Crossing under symmetry}
\label{fig:symmetry}
\end{figure}

\subsubsection{Action under Flipping}

Given a virtual knot $K$, the \textit{flip} of $K$, denoted $Flip(K)$ is defined by taking the two-dimensional plane containing knot diagram and rotating it, along with the diagram, by 180 degrees.  (Note: this action is sometimes referred to as a ``pancake flip'' as an aid for visualization.  In the first article on virtual knot theory \cite{VKT}, Kauffman calls this symmetry $Flip(K)$.
See figure \ref{fig:flipcrossing}.

\subsubsection{Action under the Switch Operation}

Given a virtual knot $K$, the \textit{switch} of $K$, denoted $K^\ast$,  is the virtual knot formed by switching all of the crossings of $K$ (see figure \ref{fig:switchcrossing}) . (Note that this operation is typically referred to as the vertical mirror image in \cite{GreenVKTables} and the horizontal mirror image in \cite{BiquandlesforVKs}. We follow the convention defined by Crans, Henrich, and Nelson \cite{CransHenrichNelson} and call it the switch to avoid confusion.) In \cite {VKT} Kauffman and \cite{BiquandlesforVKs} Hrencecin and Kauffman considers the result on the quandle and biquandle, respectively, under the swithc operation. Here we show how this operation changes the virtual parity Alexander polynomial. Note that by equations (\ref{eqn:oddrelations}) and (\ref{eqn:virtualrelations}), the relations on the odd and virtual crossings do not change. However, equations (\ref{eqn:positiveevenrelations}) and (\ref{eqn:negativeevenrelations}) produce a non-trivial change at the even crossings.

\subsubsection{Action under Switched Flip}
Given a virtual knot $K$, the \textit{switched flip} of $K$, denoted $K^\uparrow$, is the virtual knot defined as $K^\uparrow = (Flip(K))^\ast$, see figure \ref{fig:flipswitchcrossing}. (Note, in the Knot Atlas this is referred as the horizontal mirror image \cite{GreenVKTables}. However, Hrencecin and Kauffman \cite{BiquandlesforVKs} call this operation the vertical mirror image. And further, Crans, Henrich, and Nelson refer to this as the reversed inverse of the virtual knot \cite{CransHenrichNelson}.)

\begin{theorem}
Given a virtual knot diagram $K$,
\begin{enumerate}
\item  $ \Phi \Delta_K (s,t,q, \theta) =\Phi \Delta_{K'} (s,t,q, \theta) $.
\item $ \Phi \Delta_K (s,t,q, \theta) =\Phi \Delta_{K^\ast} (\frac{1}{s} ,\frac{1}{t},q, \theta)$
\item $ \Phi \Delta_K (s,t,q, \theta) =\Phi \Delta_{Flip(K)} (s ,t,\frac{1}{q},\frac{1}{ \theta})$
\item  $ \Phi \Delta_K (s,t,q, \theta) =\Phi \Delta_{K^\uparrow} (\frac{1}{s} ,\frac{1}{t},\frac{1}{q}, \frac{1}{\theta} )$
\end{enumerate}
\end{theorem}

\begin{proof}
We consider the crossing from $K$ in figure \ref{fig:standardcrossing}. From a positive even crossing, we obtain the relations
\begin{align} \label{eqn:even}
c&=sas^{-1} & d &=a b sa^{-1} s^{-1} .
\end{align}
From an odd crossing,
\begin{align}  \label{eqn:odd}
c&=qaq^{-1} & d &=q^{-1} b q.
\end{align}
Equation \ref{eqn:even}  makes the following contribution to the standard matrix
\begin{align} \label{eqn:evenmatrix}
\begin{bmatrix} s & 0 & -1 & 0 \\ 1- st & t & -1  & 0\end{bmatrix}
\end{align}
and equation \ref{eqn:odd}
makes the following contribution
\begin{align} \label{eqn:oddmatrix}
\begin{bmatrix} q & 0 & -1 & 0 \\ 0 &\frac{1}{q}& 0 & -1 \end{bmatrix} .
\end{align}
 The relations given by a virtual crossing are similar to the odd crossing.
We calculate the relations and sub-matrix for a reversed, switched, and flipped crossing.

The reverse crossing (see figure \ref{fig:reversecrossing}), has relations
\begin{align*}
a&=scs^{-1}, & b &=c d sc^{-1} s^{-1} & \text{(even)}, \\
a&=q c q^{-1}, & b &=q d q^{-1} & \text{(odd)}.
\end{align*}
The corresponding matrices are
\begin{align} \label{eqn:reverse1}
\begin{bmatrix} -1 & 0 & s & 0 \\ 0 & -1 & 1-st & t \end{bmatrix} & \text{(even)}, \\
\label{eqn:reverse2} \begin{bmatrix}
-1 & 0 & q & 0 \\
0 & -1 & 0 & q
\end{bmatrix} & \text{(odd)}.
\end{align}
The matrices in equations \ref{eqn:reverse1}  and \ref{eqn:reverse2} are related to the standard matrices (equations \ref{eqn:evenmatrix} and
\ref{eqn:oddmatrix})  by a sequence of column swaps which only changes the determinant by sign.

In the switched crossing (see figure \ref{fig:switchcrossing}),
\begin{align*}
d&=s^{-1} b s, & c &=s^{-1} b^{-1} s ab & \text{(even)}, \\
c&=q a q^{-1}, & d &=q^{-1} b q & \text{(odd)}.
\end{align*}
The corresponding matrices are
\begin{align} \label{eqn:switch1}
\begin{bmatrix} 0 & \frac{1}{s} & 0 & -1 \\ \frac{1}{t} & \frac{-1}{st} +1 & -1& 0 \end{bmatrix} & \text{(even)}, \\
\label{eqn:switch2} \begin{bmatrix}
q & 0 & -1 & 0 \\
0 & \frac{1}{q} & 0 & -1
\end{bmatrix} & \text{(odd)}.
\end{align}
These matrices in equations \ref{eqn:switch1} and \ref{eqn:switch2} are related to the standard matrices (equations
\ref{eqn:evenmatrix} and \ref{eqn:oddmatrix})  by a sequence of column swaps and the exchanges
$s \rightarrow \frac{1}{s}$ and $t \rightarrow \frac{1}{t}$.

In the flipped crossing (see figure \ref{fig:flipcrossing}),
\begin{align*}
d&=s b s^{-1}, & c &= ba s b^{-1} s^{-1}  & \text{(even)}, \\
c&=q^{-1} a q, & d &=q b q^{-1} & \text{(odd)}.
\end{align*}
The corresponding matrices are
\begin{align} \label{eqn:flip1}
\begin{bmatrix} 0 & s & 0 & -1 \\ t & 1-st  & -1& 0 \end{bmatrix} & \text{(even)}, \\
\label{eqn:flip2} \begin{bmatrix}
\frac{1}{q} & 0 & -1 & 0 \\
0 & q  & 0 & -1
\end{bmatrix} & \text{(odd)}.
\end{align}
The matrices in equations \ref{eqn:flip1} and \ref{eqn:flip2} are related to the standard matrices (equations \ref{eqn:evenmatrix} and \ref{eqn:oddmatrix}) by a sequence of column swaps and the exchanges
$q \rightarrow \frac{1}{q}$ and $\theta \rightarrow \frac{1}{\theta}$.
\end{proof}

\section{Conclusion}
In future work, we plan to consider the applications of odd Vassiliev invariants and their strength. This result suggests that virtual knot diagrams can not be unknotted by crossing change on odd crossings except in specialized circumstances. We make the following conjecture about even Vassiliev invariants. An even, degree one Vassiliev invariant, $\it{V}$ has the property that $ \it{V}(K_+) - \it{V} (K_-) =0 $ for even crossings and  $ \it{V}(K_+) - \it{V} (K_-) \neq 0 $ for odd crossings.
\begin{conjecture} There are no even Vassilliev invariants. \end{conjecture}

\bibliographystyle{plain}
\bibliography{VPApolynomial}

\end{document}